\documentclass[10pt]{amsart}
\usepackage{amsmath, amscd, amsthm} 
\usepackage{amssymb, amsfonts} 
\usepackage{color}
\usepackage[all]{xy} 
\usepackage{marginnote}
\usepackage{aliascnt}
\usepackage[svgnames]{xcolor} 
\usepackage{hyperref}
\usepackage{dsfont} 
\definecolor{dark-red}{rgb}{0.4,0.15,0.15}
\hypersetup{
	colorlinks, linkcolor=dark-red,
	citecolor=DarkBlue, urlcolor=MediumBlue
}

\definecolor{refkey}{gray}{0.5}
\definecolor{labelkey}{gray}{0.5}

\newcommand{\A}{\mathds{A}}
\renewcommand{\P}{\mathds{P}}
\newcommand{\Z}{\mathds{Z}}

\newcommand{\G}{\mathds{G}}

\newcommand{\ZZ}{\mathcal{Z}}
\newcommand{\HH}{\mathrm{H}}
\newcommand{\VV}{\mathcal{V}}
\newcommand{\WW}{\mathcal{W}}

\newcommand{\et}{\acute{e}t}

\newcommand{\iso}{\cong}

\newcommand{\mcal}[1]{\mathcal{#1}}

\DeclareMathOperator*{\colim}{\mathrm{colim}}

\DeclareMathOperator{\spec}{\mathrm{Spec}}
\DeclareMathOperator{\proj}{\mathrm{Proj}}

\DeclareMathOperator{\supp}{Supp}
\newcommand{\Sym}{\mathrm{Sym}}
\newcommand{\Div}{\mathrm{Div}}
\newcommand{\Pic}{\mathrm{Pic}}

\newcommand{\Ab}{\mathrm{Ab}}
\newcommand{\Sus}{\mathrm{Sus}}
\newcommand{\red}{\mathrm{red}}

\newcommand{\Vect}{\mathrm{Vect}}

\DeclareMathOperator{\Hom}{Hom}

\DeclareMathOperator{\Cor}{Cor}
\DeclareMathOperator{\Aut}{Aut}
\DeclareMathOperator{\ch}{char}

\DeclareMathOperator{\tr}{Tr}

\newcommand{\Sch}{\mathrm{Sch}}
\newcommand{\Sm}{\mathrm{Sm}}

\numberwithin{equation}{section} 

\theoremstyle{plain}

\newaliascnt{theorem}{equation}  
\newtheorem{theorem}[theorem]{Theorem}  
\aliascntresetthe{theorem}

\newaliascnt{proposition}{equation}  
\newtheorem{proposition}[proposition]{Proposition}
\aliascntresetthe{proposition}

\newaliascnt{lemma}{equation}  
\newtheorem{lemma}[lemma]{Lemma}
\aliascntresetthe{lemma}

\newaliascnt{corollary}{equation}  
\newtheorem{corollary}[corollary]{Corollary}
\aliascntresetthe{corollary}

\newaliascnt{claim}{equation}  

\aliascntresetthe{claim}

\newaliascnt{conjecture}{equation}  

\aliascntresetthe{conjecture}

\theoremstyle{definition}

\newaliascnt{definition}{equation}  
\newtheorem{definition}[definition]{Definition}
\aliascntresetthe{definition}

\newaliascnt{example}{equation}  
\newtheorem{example}[example]{Example}
\aliascntresetthe{example}

\newaliascnt{remark}{equation}  
\newtheorem{remark}[remark]{Remark}
\aliascntresetthe{remark}

\newcommand{\aref}[1]{\autoref{#1}}

\begin{document}
\title{Rigidity for Equivariant Pseudo Pretheories}
\author{Jeremiah Heller}
\address{University of Illinois Urbana-Champaign}
\email{jbheller@illinois.edu}
\author{Charanya Ravi}
\address{University of Oslo}
\email{charanyr@math.uio.no}
\author{Paul Arne {\O}stv{\ae}r}
\address{University of Oslo}
\email{paularne@math.uio.no}

\begin{abstract} 
We prove versions of the Suslin and Gabber rigidity theorems in the setting of equivariant pseudo pretheories of smooth schemes over a field with an action of a finite group.
Examples include equivariant algebraic $K$-theory, presheaves with equivariant transfers, equivariant Suslin homology, and Bredon motivic cohomology. 
\end{abstract}
\maketitle

\section{Introduction} \label{sec:intro}

The classical rigidity theorems for algebraic $K$-theory are due to Suslin \cite{Sus83} for extensions of algebraically closed fields, 
Gabber \cite{Gab92} for Hensel local rings, 
and Gillet-Thomason \cite{GT84} for strictly Hensel local rings. 
All known proofs rely on $\A^{1}$-homotopy invariance and existence of transfer maps with certain nice properties.
In his work on motives,  
Voevodsky introduced homotopy invariant pretheories as contravariant functors on smooth schemes over a field enjoying certain transfer maps \cite[Definition~3.1]{MR1764200}.
While algebraic $K$-theory admits transfer maps for relative smooth curves, it is not an example of a pretheory \cite[\S3.4]{MR1764200}.
However, 
it is the motivating example of a pseudo pretheory in the sense of Friedlander-Suslin \cite[Section~10]{FS02}.
The work of Suslin-Voevodsky \cite{SV96} established rigidity theorems in the context of homotopy invariant pseudo pretheories.

In this paper, we generalize the notion of pseudo pretheories to the equivariant setting of finite group actions (\aref{def:pseudo}).
Equivariant algebraic $K$-theory is an example, as well as equivariant Suslin homology, and Bredon motivic cohomology in the sense of  \cite[Section~5]{HVO}.

Our main results establish equivariant analogs of the Suslin-Voevodsky rigidity theorems in \cite[Section~4]{SV96} (see  \aref{thm:suslinrigid}, \aref{thm:gabberrigid}). 

\begin{theorem} \label{thm:Mainthm}
Let $k$ be a field, 
$G$ be a finite group whose order is invertible in $k$, 
and let $\Sm^G_k$ denote the category of smooth schemes over $k$ equipped with an action of $G$.
Let $F$ be a homotopy invariant equivariant pseudo pretheory on $\Sm^G_k$. 
Suppose that $F$ is torsion of exponent coprime to $\ch(k)$. 
\begin{enumerate}
\item 
Let $S = \spec(\mathcal{O}^h_{W,Gw})$ be the Henselization of a smooth affine $G$-scheme 
$W$ at the orbit $Gw$ of a closed point.
Let $X \to S$ be a smooth affine G-scheme of relative dimension one, 
admitting an equivariant good compactification. 
Then for all equivariant sections $i_1, i_2 : S \to X $ which coincide on the closed orbit of $S$, 
we have 
$$
i_1^* = i_2^* : F(X) \to F(S).
$$  
\item 
Let $X$ be a smooth affine $G$-scheme and let $x\in X$ be a closed point
such that $k \subseteq k(x)$ is separable.
If every representation of $G$ over $k$ is a direct sum of one dimensional representations,
then there is a naturally induced isomorphism
$$
F(Gx) \xrightarrow{\iso} F(\spec(\mcal{O}_{X,Gx}^{h})).
$$ 
\end{enumerate}
\end{theorem}

The condition in the second part of the theorem is satisfied whenever $G$ is abelian and $k$ contains a primitive $d$th root of unity, where $d$ is the exponent of the group, by a theorem of Brauer, see e.g., \cite[Theorem 41.1, Corollary 70.24]{CR}.

Rigidity theorems have been established for equivariant algebraic $K$-theory in \cite{YO09} and \cite[Theorem~1.4]{Kri10} at points with trivial stabilizers.
The novelty in \aref{thm:Mainthm} is that we allow points with \emph{nontrivial} stabilizers.
Note,
however, 
that in \cite{YO09} the groups are more general, 
and \cite{Kri10} deals with connected split reductive groups.
For works on rigidity results in related contexts, 
see e.g., 
\cite{AD}, \cite{MR3205601}, \cite{MR3477640}, \cite{HY07}, \cite{UJ}, \cite{MR3443257}, \cite{PY02}, \cite{RO06}, \cite{MR2399164}, \cite{Tab}, and \cite{MR2854332}.

A brief overview of the paper follows.
 \aref{sec:Prelim} recalls notions in $G$-equivariant algebraic geometry and shows an equivariant proper base change theorem for \'etale cohomology of Henselian pairs.
After recalling equivariant divisors and equivariant correspondences, 
we define and give examples of equivariant pseudo pretheories in  \aref{sec:EqDiv}.
Next in   \aref{sec:GComp} we discuss the equivariant Nisnevich topology and equivariant good compactification for smooth affine relative curves.
Our main results are shown in   \aref{sec:rigidity}.
Finally, 
in  \aref{sec:Gerst-resol} we show that exactness of the Gersten complex for equivariant algebraic $K$-theory fails for the group $G=\Z/2\Z$ of order two acting on the
affine line $\A^1_k = \spec(k[t])$ by $t \mapsto -t$.
This follows by applying rigidity to the $G$-equivariant Grothendieck group $K_0^G$ of the Henselization $\mcal{O}^h_{\A^1_k,Gx}$ at the orbit of the closed point $x=(t)\in \A^1_k$.

\subsection*{Acknowledgements.}
Work on this paper took place at the Institut Mittag-Leffler during Spring 2017.
We thank the institute for its hospitality and support.
The authors gratefully acknowledge funding from the RCN Frontier Research Group Project no. 250399 ``Motivic Hopf equations." 
Heller is supported by NSF Grant  DMS-1710966.
{\O}stv{\ae}r is supported by a Friedrich Wilhelm Bessel Research Award from the Humboldt Foundation and a Nelder Visiting Fellowship from Imperial College London.
The authors would like to thank the referee for a careful reading of this paper and for an 
insightful comment about the equivariant Gersten complex, which is included in the text as 
Remark \ref{rem:Gersten}.

\section{Preliminaries} 
\label{sec:Prelim}
Throughout $k$ is a field and $G$ is a finite group whose order is coprime to $\ch(k)$
(abusing the terminology we say that $n$ is coprime to $\ch(k)$ if 
$n$ is coprime to the exponential characteristic of $k$, i.e., $n$ is invertible in $k$).
We view $G$ as a group scheme  $\coprod_G \spec(k)$ over $\spec(k)$. 
Let $\Sch^G_k$ be the category of separated, finite type schemes over $\spec(k)$ equipped with a left $G$-action, and equivariant morphisms. 
The smooth $G$-schemes over $\spec(k)$ form a full subcategory $\Sm^G_k\subseteq \Sch^G_k$.
A $G$-scheme $X$ is equivariantly irreducible if there exists an irreducible component $X_0$ of $X$ such that $G \cdot X_0 = X$. 
The fiber product $X \times Y$ of $X, Y \in \Sch^G_k$ is a $G$-scheme with the diagonal $G$-action.
For a finite dimensional $k$-vector space $V$, 
let $\A(V):= \spec(\Sym(V^\vee))$ and $\P(V) := \proj(\Sym(V^\vee))$. 
If $V$ is a $G$-representation over $k$,
we view $\A(V)$ and $\P(V)$ as $G$-schemes via the $G$-action on $V$.

For $X\in\Sch^G_k$ we denote the categorical quotient of $X$ by $G$ (in the sense of \cite[Definition~0.5]{MFK}) by $X/G$, provided it exists.
Since $G$ is a finite group, the categorical quotient map $\pi: X \to X/G$ is in fact a uniform geometric quotient (\cite[Definitions~0.6, 0.7]{MFK}).
If $X$ is quasi-projective, then a quotient by a finite group $\pi: X \to X/G$ always exists.

Let $H\subseteq G$ be a subgroup and $X\in\Sch^H_k$. 
Then $G\times X$ is an $H$-scheme with the action $h(g,x) = (gh^{-1}, hx)$, 
and we define $G \times^{H} X := (G \times X)/H$.
The scheme $G \times^{H} X$ has a left $G$-action through the action of $G$ on itself. 
Since the $H$-action on $G \times X$ is free, $\pi: G \times X \to G \times^{H} X$ is a principle $H$-bundle. 
In particular, 
$\pi$ is \'etale and surjective. 
It follows that if $X$ is smooth, 
then so is $G \times^{H} X$. 
This defines a left adjoint to the restriction functor $\Sm^G_k \to \Sm^H_k$,
given by $G \times^{H}- : \Sm^H_k \to \Sm^G_k$.

For $X\in\Sch^G_k$ and $x \in X$ a point, the {\sl set-theoretic stabilizer} of $x$ is the subgroup $G_x \subseteq G$ defined by $G_x = \{g \in G| g·x = x \}$.
The {\sl orbit} of $x$ is $G·x := G \times^{G_x} \{x\}$, 
with underlying set $\{g·x | g \in G\}$.

\subsection{{$G$}-sheaves} \label{sub:Gshv}

A $G$-sheaf on $X$ is basically a sheaf with a $G$-action which is compatible with the $G$-action on $X$. 
The precise definition goes as follows.

\begin{definition} \label{def:GShv}
Let $\tau$ be a Grothendieck topology on $X$ and $\mcal{F}$ a $\tau$-sheaf of abelian groups. 
Write ${\rm pr}_{2}:G\times X\to X$ for the projection and $\mu:G\times X\to X$ for the action map.
\begin{enumerate}
\item 
A \emph{$G$-linearization} of $\mcal{F}$ is an isomorphism
$
\phi:\mu^{*}\mcal{F}\xrightarrow{\iso} {\rm pr}_{2}^{*}\mcal{F}
$
of sheaves on $G\times X$ which satisfies the cocycle condition  
$
{\rm pr}_{23}^{*}(\phi)\circ(Id_G \times\mu)^{*}(\phi) = (m\times Id_X)^{*}(\phi)
$ 
on $G\times G \times X$. 
Here $m:G\times G\to G$ is the multiplication and ${\rm pr}_{23}:G\times G\times X \to G\times X$ is the projection to second and third factors. 
\item 
A \textit{{$G$}-sheaf} (in the $\tau$-topology) on $X$ is a pair consisting of a $\tau$-sheaf $\mcal{F}$ together with a $G$-linearization $\phi$ of $\mcal{F}$. 
We simply write $\mcal{F}$ for a $G$-sheaf,
leaving the $G$-linearization understood.
\item 
A \textit{{$G$}-module} $\mcal{M}$ on $X$ is a $G$-sheaf  
where $\mcal{M}$ is a quasi-coherent $\mcal{O}_{X}$-module and the $G$-linearization $\phi:\mu^{*}\mcal{M}\iso pr_{2}^{*}\mcal{M}$ is an  $\mcal{O}_{G\times X}$-module isomorphism. 
A \textit{$G$-vector bundle} on $X$ is a $G$-module  $\mcal{V}$ whose underlying quasi-coherent $\mcal{O}_{X}$-module is locally free. 
\end{enumerate}
\end{definition}

\begin{remark}
Since $G$ is finite, 
the data of a $G$-linearization of $\mcal{F}$ is equivalent to giving a sheaf isomorphism $\phi_g : \mcal{F} \xrightarrow{\iso} g_*\mcal{F}$ for each $g \in G$ subject to the conditions $\phi_e = id$
and $\phi_{gh} = h_*(\phi_g) \circ \phi_h$ for all $g, h \in G$.
\end{remark}
\begin{remark}\label{rem:skew}
Recall that if $G$ acts on a commutative ring $R$, the skew group ring $R\wr G$ is the free left $R$-module with basis $\{[g]\,|\, g\in G\}$ and multiplication is defined by setting
$(r[g])(s[h]) = r(g\smash\cdot s)[gh]$ and extending linearly. If $G$ acts trivially on $R$, then $R\wr G$ is simply the usual group ring $RG$.

If $X=\spec(R)$, then the category of $G$-modules on $X$ is equivalent to the category of left $R\wr G$-modules. Further, if the order of $G$ is invertible in $R$, then the category of $G$-vector bundles on $X$ is equivalent to the category of left $R\wr G$-modules which are projective as $R$-modules.  See e.g., \cite[Section 1.1]{LS08} for details.
\end{remark}

A {\emph{$G$-equivariant morphism} $f:(\mcal{E},\phi_{\mcal{E}}) \to (\mcal{F},\phi_{\mcal{F}})$ of $G$-sheaves is a morphism $f: \mcal{E} \to \mcal{F}$ of sheaves compatible with the 
$G$-linearizations in the sense that $\phi_{\mcal{F}}\circ \mu^{*}f = pr_{2}^{*}f\circ \phi_{\mcal{E}}$, 
or equivalently $\phi_g \circ f = g_*(f) \circ \phi_g$ for all $g \in G$.
Write $Ab_{\tau}(G,X)$ for the category of $G$-sheaves on $X$ in the $\tau$-topology. 
We note that $Ab_{\tau}(G,X)$ has enough injectives. 

Given a $G$-sheaf $(\mcal{F}, \phi_g)$, 
the morphisms $\phi_g$ induce an action of the group $G$ on the group of global sections $\Gamma(X,\mcal{F})$.
We write $\Gamma^{G}_{X}(\mcal{F}) = \Gamma(X,\mcal{F})^{G}$ for the set of $G$-invariants of $\Gamma(X,\mcal{F})$. 
This defines a functor $\Gamma^{G}_{X}:\Ab_{\tau}(G,X) \to \Ab$ from the category of $G$-sheaves to the category of abelian groups. 
The $\tau$-$G$-cohomology groups $H^{p}_{\tau}(G;X,\mcal{M})$ are defined as right derived functors
$$
H^{p}_{\tau}(G;X, \mcal{F}) := R^{p}\Gamma^{G}_{X}(\mcal{F}).
$$
Here $\Gamma^{G}_{X} = (-)^{G}\circ \Gamma(X,-)$ is a composite of left exact functors. 
Since the global sections functor $\Gamma(X,-)$ sends injective $G$-sheaves to injective $\Z[G]$-modules, 
the Grothendieck spectral sequence for this composition yields the bounded, convergent spectral sequence 
\begin{equation}
\label{eqn:ss}
E_{2}^{p,q} = H^{p}(G, H^{q}_{\tau}(X,\mcal{F})) \Rightarrow H^{p+q}_{\tau}(G;X,\mcal{F}),
\end{equation}
where $H^{*}(G, - )$ denotes the group cohomology of $G$. Moreover, the spectral sequence
induces a finite filtration on each $H^{n}_{\tau}(G;X,\mcal{F})$.

\begin{definition} 
\label{def:eqpic}
The {\sl $G$-equivariant Picard group} $\Pic^G(X)$ of $X$ is the group of $G$-line bundles on $X$ modulo equivariant isomorphisms, 
with group operation given by tensor product. 
For an invariant closed subscheme $Y \subseteq X$, 
let $\Pic^G(X,Y)$ denote the group consisting of pairs $(\mcal{L}, \phi)$, 
where $\mcal{L}$ is a $G$-line bundle on $X$ and $\phi: \mcal{O}_Y \xrightarrow{\cong} \mcal{L}|_Y$ is an isomorphism of $G$-line bundles on $Y$, 
modulo equivariant isomorphisms respecting the trivializations on $Y$. 
The group $\Pic^G(X,Y)$ is called the {\sl relative equivariant Picard group} of $X$ relative to $Y$.
\end{definition}

The following cohomological interpretations of the equivariant and the relative equivariant Picard groups are standard, 
see \cite[Theorem 2.7, Lemma 6.7]{HVO}.

\begin{theorem}\label{thm:H90}
Let $X$ be a $G$-scheme. 
\begin{enumerate}
\item 
There is a natural isomorphism
$
\Pic^{G}(X) \xrightarrow{\iso} H^{1}_{\et}(G;X,\mcal{O}^{*}_{X}).
$
\item Let $i:Y\hookrightarrow X$ be an invariant closed 
subscheme. Then there is a natural isomorphism
$
\Pic^{G}(X,Y) \xrightarrow{\iso} H^{1}_{\et}(G;X,\G_{X,Y}),
$
where $\G_{X,Y}$ is the \'etale $G$-sheaf defined to be the kernel of the equivariant homomorphism $\mcal{O}^{*}_{X} \to i_{*}\mcal{O}^{*}_{Y}$.
\end{enumerate}
\end{theorem}

We end this section by recording an equivariant version of Gabber's proper base change theorem for the cohomology of torsion \'etale $G$-sheaves, 
which will be needed to establish the equivariant version of Suslin's rigidity theorem in   \aref{sec:rigidity}. 

\begin{definition} 
\label{def:hensel} (\cite[Chapter XI, Definition 3]{Ray06})
Let $A$ be a commutative ring and $I\subseteq A$ an ideal which is contained in the Jacobson radical of $A$. 
The pair $(A,I)$  is said to be a \textit{Henselian pair} provided $\Hom_{A}(B,A)\to \Hom_{A}(B,A/I)$ is surjective for any \'etale $A$-algebra $B$.
 A $G$-action on a Henselian pair $(A, I)$ is simply a $G$-action on $A$ such that the ideal $I$ is invariant.
\end{definition}

\begin{theorem}[Equivariant Proper Base Change]
\label{thm:basechange}
Let $(A,I)$ be a Henselian pair with $G$-action. 
Let $f:Y\to \spec(A)$ be a proper equivariant map and define $Y_{0}$ by the pull-back
$$
\xymatrix{
Y_{0} \ar[r]^{i}\ar[d]_{f'} & Y \ar[d] ^{f} \\
\spec(A/I) \ar[r]^{j} & \spec(A).
}
$$
Let $\mcal{F}$ be a torsion \'etale $G$-sheaf on $Y$ and write $\mcal{F}_{0} = i^{*}\mcal{F}$.  
Then the restriction map induces an isomorphism $H^{n}_{\et}(G;Y,\mcal{F}) \iso H^{n}_{\et}(G;Y_{0}, \mcal{F}_{0})$ for each $n$. 
\end{theorem}
\begin{proof} 
Restriction induces a $G$-equivariant map $H^{p}_{\et}(Y,\mcal{F}) \to H^{p}_{\et}(Y_{0},\mcal{F}_{0})$.
Gabber's base change theorem \cite[Corollary 1]{Gab94} shows this is an isomorphism, and therefore it induces an isomorphism in group cohomology. 
Thus the induced comparison maps of spectral sequences (\ref{eqn:ss}) for $(Y, \mcal{F})$ and $(Y_0, \mcal{F}_0)$ is an isomorphism on the $E_{2}$-page. 
This implies the desired isomorphism.

\end{proof}

\section{Equivariant divisors and pseudo pretheories} 
\label{sec:EqDiv}
We begin by recalling the notion of equivariant Cartier divisors and their properties.
\subsection{Equivariant divisors}
Let $X$ be a $G$-scheme and $Y\subseteq X$ an invariant closed subscheme. 
\begin{definition} \label{def:eqDiv}
\begin{enumerate}
\item
An \textit{equivariant Cartier divisor} on $X$ is an element of 
$\Gamma^G_X(\mcal{K}_{X}^{*}/\mcal{O}_{X}^{*})$. 
The group of equivariant Cartier divisors on $X$ is denoted by $\Div^{G}(X)$.
An effective Cartier divisor $D$ on $X$ such that $D \in \Gamma^G_X(\mcal{K}_{X}^{*}/\mcal{O}_{X}^{*})$ is called an {\it equivariant effective Cartier divisor}.
\item
A \textit{relative equivariant Cartier divisor} on $X$ relative to $Y$ is an equivariant Cartier divisor $D$ on $X$ such that $\supp(D)\cap Y = \emptyset$. 
Write $\Div^{G}(X,Y)$ for the subgroup of $\Div^{G}(X)$ consisting of relative equivariant Cartier divisors.
\item
A \textit{principal equivariant Cartier divisor} is an invariant rational function on $X$, 
i.e., 
an element in the image of $\Gamma^G_X(\mcal{K}_{X}^{*})$ in $\Gamma^G_X(\mcal{K}_{X}^{*}/\mcal{O}_{X}^{*})$. 
In the relative setting, 
a \textit{principal equivariant Cartier divisor} $f$ on $X$ is said to be a \textit{principal relative equivariant Cartier divisor} if $f$ is defined and equal to $1$ at points of $Y$.
\item
Let $\Div^{G}_{rat}(X)$ denote the group of equivariant Cartier divisors on $X$ modulo the principal equivariant Cartier divisors, 
and likewise write $\Div^{G}_{rat}(X,Y)$ in the relative setting.
\end{enumerate}
\end{definition}

Given a Cartier divisor $D= \{(U_{i},f_{i})\}$ on $X$, we have an associated line bundle $\mcal{L}_D$ defined by $\mcal{L}_D|_{U_{i}} = \mcal{O}_{U_{i}}f_{i}^{-1}$.
When $D$ is an equivariant Cartier divisor it is easy to verify that the line bundle $\mcal{L}_D$ has a canonical $G$-linearization; 
write $\mcal{L}_D$ for the $G$-line bundle defined by this choice of linearization.
If $D$ is a relative equivariant Cartier divisor relative to $Y$ it is straightforward that $\mcal{L}_D|_Y$ is trivial.

Let $\mcal{Z}_{d}(X)$ (respectively $\mcal{Z}^{d}(X)$) denote the free group on dimension $d$ (respectively codimension $d$) cycles on $X$.
The homomorphism $cyc: Div(X) \to \mcal{Z}^{1}(X)$ is defined by $cyc(D) = \sum_{Z\in X^{1}}ord_{Z}(D)Z$, 
where $X^{1}$ is the set of closed integral codimension one subschemes. 
For a $G$-scheme $X$, 
the groups $\mcal{Z}_{d}(X)$ and $\mcal{Z}^{d}(X)$ have natural $G$-actions and $cyc$ is an equivariant homomorphism.
Therefore we conclude the following.

\begin{lemma} 
\label{lem:divagr}(\cite[Lemma~2.11]{HVO})
For a smooth $G$-scheme $X$,
$
cyc:\Div(X) \to \mcal{Z}^{1}(X)
$ 
is an equivariant isomorphism. 
\end{lemma} 

\subsection{Equivariant pseudo pretheories} 
\label{sub:pseudo}
An equivariant pseudo pretheory is defined as a presheaf on $\Sm^G_k$ with transfer maps associated to certain equivariant correspondences subject to some natural axioms.

\begin{definition} 
\label{def:pseudo}
An \textit{equivariant pseudo pretheory} on $\Sm^G_k$ is an additive presheaf $F:(\Sm^G_k)^{op}\to \Ab$ 
(i.e., $F(X\coprod Y) = F(X) \oplus F(Y)$) with transfer maps $\tr_{D}:F(X)\to F(S)$ for any equivariant relative smooth affine curve $X/S$ and effective equivariant Cartier divisor 
$D$ on $X$ which is finite and surjective over a component of $S$, such that the following holds.
\begin{enumerate}
\item 
The transfer maps are compatible with pullbacks.
\item 
If $D(i)$ is the divisor associated to an equivariant section $i:S\to X$, then 
$$
\tr_{D(i)} = F(i).
$$
\item 
Let $\mcal{L}_{D}$ be the $G$-line bundle associated to $D$. If the restriction of $\mcal{L}_{D}$ to $D'$ is trivial, then
$$
\tr_{D} + \tr_{D'} = \tr_{D + D'}.
$$
\end{enumerate}
\end{definition}

As usual we extend all functors defined on the category $\Sm^G_k$ to limits of smooth 
$G$-schemes with $G$-action (including semilocalizations of all smooth affine $G$-schemes at closed $G$-orbits) by taking direct limits. 
The above properties obviously remain true after such an extension as well.

\begin{definition}
A presheaf $F$ on $\Sm^G_k$ (or $\Sch^G_k$) is said to be \textit{homotopy invariant} if for any $X \in \Sm^G_k$ (respectively in $\Sch^G_k$) the projection map
$p_1:X\times\A^{1}\to X$ induces an isomorphism $p_1^{*}:F(X)\xrightarrow{\iso} F(X\times \A^{1})$,
where the $G$-action on $X\times \A^{1}$ is induced by the given $G$-action on $X$ and the trivial $G$-action on $\A^{1}$.
\end{definition}

\subsection{Examples of equivariant pseudo pretheories} 
\label{sub:examples}
In the following we discuss examples of equivariant pseudo pretheories such as equivariant algebraic $K$-theory, 
equivariant Suslin homology, 
$K_{0}^{G}$-presheaves with transfers, 
presheaves with equivariant transfers, 
and equivariant motivic representable theories. 

\begin{example}
\label{ex:presheavesequivarianttransfers} {\bf Presheaves with equivariant transfers.}
For smooth schemes $X$, $Y$,
the group of correspondences $\Cor_{k}(X,Y)\subseteq \mcal{Z}_{\dim(X)}(X\times Y)$ is the subgroup of $\mcal{Z}_{\dim(X)}(X\times Y)$ of cycles on $X\times Y$ 
which are finite over $X$ and surjective over some component of $X$. 
The category $\Cor_{k}$ has the same objects as $\Sm/k$ and $\Cor_{k}(X,Y)$ are the morphisms between $X$ and $Y$ in this category. 
The {\sl equivariant correspondences} $\Cor^{G}_{k}(X,Y)$ between smooth $G$-schemes are correspondences $Z:X\to Y$ such that the square
$$
\xymatrix{
G\times X  \ar[r]^{Z \times id}\ar[d]_{\mu} & G\times Y \ar[d]^{\mu}\\
X \ar[r]^{Z} & Y
}
$$
commutes in $\Cor_{k}$ \cite[Section~4]{HVO}. 
Unravelling definitions we have 
$$
\Cor^{G}_{k}(X,Y) = \Cor_{k}(X,Y)\cap \ZZ_{\dim X}(X\times Y)^{G}.
$$
Let $\Cor^{G}_{k}$ denote the category whose objects are smooth $G$-schemes and morphisms are equivariant correspondences.
There is a canonical inclusion $\Sm^G_k\subseteq \Cor^{G}_{k}$ which sends $f:X\to Y$ to its graph $\Gamma_{f}\subseteq X\times Y$.

\begin{definition} \label{def:psht}  \cite[Definition 4.1]{HVO}
A {\sl presheaf with equivariant transfers} is a presheaf of abelian groups on the category $\Cor^{G}_{k}$.
\end{definition}

Given an equivariant relative smooth affine curve $X/S$ and an 
effective equivariant Cartier divisor 
$D$ on $X$ which is finite and surjective over $S$, note that $D \in \Cor^{G}_{k}(S,X)$.
Moreover, if $D(i)$ is the divisor associated to an equivariant section $i: S \to X$,
then $D(i) = \Gamma_i$ in $\Cor^{G}_{k}(S,X)$.
Therefore if $F$ is a presheaf with equivariant transfers, then $F$ defines an additive presheaf
on $\Sm^G_k\subseteq \Cor^{G}_{k}$ such that for a divisor $D$ as above,
$\tr_D := F(D): F(X) \to F(S)$ satisfies conditions $(1)$, $(2)$ and $(3)$ of 
 \aref{def:pseudo}. 
\end{example}

\begin{example}
\label{ex:K^G} {\bf Equivariant $K$-theory.} 
The $G$-equivariant algebraic $K$-theory group $K_i^G(X)$ of a scheme $X$ with $G$-action is the $i$th homotopy group of the algebraic $K$-theory spectrum $K^G(X)$ of the exact category 
of $G$-vector bundles on $X$.
For $n\geq 2$, 
the equivariant $K$-groups with mod-$n$ coefficients are defined as $K_i^G(X;n):= \pi_i(K^G(X) \wedge \mathbb{S}/n)$, 
for the mod-$n$ Moore spectrum $\mathbb{S}/n$.

The equivariant algebraic $K$-theory groups $K_i^G$ define functors on $\Sch^G_k$ (and $\Sm^G_K$) by considering the category of ``big $G$-vector bundles" (\cite[Appendix C.4, C.5]{FS02}).
Let $p: X \to S$ be an equivariant relative smooth affine curve in $\Sm^G_k$ and let $i_D: D \hookrightarrow X$ be an effective equivariant Cartier divisor on $X$ such that $p_D:= p|_D: D \to S$
is finite and surjective. 
Then $p_D : D \to S$ is also flat. 
Let $\tr_D : K_i^G(X) \to K_i^G(S)$ denote the map induced by the functor $F_D: \Vect^G(X) \to \Vect^G(S)$ between the categories of $G$-vector bundles on $X$ and $S$ defined by 
$P\mapsto p_{D_*} \circ i_D^* (P)$.
By \cite[Theorem~4.1, Corollary~5.8(2)]{T87}, $K_i^G$ is a homotopy invariant functor on $\Sm^G_k$.
We show that $K_i^G$ is an equivariant pseudo pretheory on $\Sm^G_k$, 
so that $K_i^G(-;n)$ is a homotopy invariant equivariant pseudo-pretheory on $\Sm^G_k$ with $n$-torsion values.
\end{example}

\begin{lemma} 
\label{lem:K^G-pseudo-pre}
If $D$ and $D'$ are effective equivariant Cartier divisors on $X$ such that the restriction of the $G$-line bundle $\mcal{L}_{D}$ to the $G$-scheme $D'$ is a trivial $G$-line bundle,
then $\tr_{D + D'} = \tr_{D} + \tr_{D'}$.
\end{lemma}
\begin{proof}
We write $i: D \hookrightarrow D + D'$ and $i': D' \hookrightarrow D + D'$ for the corresponding 
$G$-equivariant closed immersions.
Let $f \in \Gamma^G_{D'}(\mcal{L}_D|_{D'})$ define the trivialization of  $\mcal{L}_D$ on $D'$.
Since $\mcal{L}_D$ defines the ideal sheaf of $D$,
we have an exact sequence of $G$-equivariant coherent sheaves on $D + D'$:
\begin{equation} \label{eqn:EPP1}
0 \rightarrow i'_*(\mcal{O}_{D'}) \xrightarrow{f}  \mcal{O}_{D+D'} \rightarrow i_*(\mcal{O}_D) 
\rightarrow 0,
\end{equation}
where the maps are $G$-equivariant.
Given $P \in \Vect^G(X)$, 
the above exact sequence gives the following exact sequence:
$$
0 \rightarrow i'_*\circ i_{D'}^*(P) \rightarrow  
i_{D+D'}^*(P) \rightarrow 
i_* \circ i_D^*(P) \rightarrow 0.
$$
Pushforward by the equivariant, finite, and flat map $p_{D+D'}$ gives an exact sequence of 
$G$-vector bundles on $S$:
$$
0 \rightarrow p_{D'_*} \circ i_{D'}^*(P) \rightarrow  
p_{D+D'_*} \circ i_{D+D'}^*(P) \rightarrow 
p_{D_*} \circ i_D^* (P) \rightarrow 0,
$$
which by definition of the transfer maps is the exact sequence of functors:
$$
0 \rightarrow\tr_{D'}(P) \rightarrow  
\tr_{D+D'}(P)\rightarrow 
\tr_{D}(P)\rightarrow 0.
$$
Therefore by Waldhausen's additivity theorem, 
\cite[Proposition 1.3.2(4)]{W85}, 
we conclude that $\tr_{D + D'} = \tr_{D} + \tr_{D'}$.
\end{proof}

\begin{example}
\label{sub:SusHom}
{\bf Equivariant Suslin Homology.} 
For $n \in \mathbb{N}$, the algebraic $n$-simplex $\Delta^n$ is 
$$
\Delta^n 
:= 
\spec\left(\frac{k[t_0, \cdots , t_n]}{(\sum_i t_i - 1)}\right)
$$
and $\Delta^{\bullet} = \{\Delta^n\}_{n \ge 0}$ is a cosimplicial scheme with face and degeneracy maps given by:
\[
\partial_r(t_j) = \left\{ \begin{array}{ll}
t_j  & \mbox{if $j < r$} \\
0  & \mbox{if $j = r$} \\
t_{j-1} & \mbox{if $j > r$}
\end{array} \right. \quad
\delta_r(t_j) = \left\{ \begin{array}{ll}
t_j  & \mbox{if $j < r$} \\
t_j + t_{j+1}  & \mbox{if $j = r$} \\
t_{j+1} & \mbox{if $j > r$.}
\end{array} \right.
\]
We view $\Delta^{\bullet}$ as a cosimplicial $G$-scheme with trivial $G$-action.

For a smooth morphism $f:X \to S$, let $C_0(X/S) \subseteq \Cor_{k}(S,X)$ denote the group of cycles on $X$ which are finite and surjective over a component of $S$.
If $X, S \in \Sch^G_k$ and $f$ is $G$-equivariant, then $C_0(X/S)$ is a $G$-invariant subset of $\Cor_{k}(S,X)$.
We let $C_{\bullet}(X/S)^{G}$ denote the chain complex associated to the simplicial abelian group $n\mapsto C_{n}(X/S)^{G}$, 
where $C_{n}(X/S) := C_{0}(X\times\Delta^{n}/S\times\Delta^{n})$.

\begin{definition} 
\label{def:SusHom}
The $n$th \textit{equivariant Suslin homology} of $X/S$ is defined as the $n$th homology group of the complex of abelian groups $C_{\bullet}(X/S)^{G}$:
$$
\HH_{n}^{\Sus}(G;X/S) := H_{n}C_{\bullet}(X/S)^{G}.
$$
\end{definition}

For a smooth $G$-scheme $X$ over $k$,
let $\Z_{tr,G}(X)$ denote the presheaf with equivariant transfers given by the representable 
functors $\Z_{tr,G}(X)(U) := C_{0}(X \times U/U)^G = \Cor_k^G(U,X)$ 
for each $U \in \Sm^G_k$.
When $G$ is trivial, this is the same as the presheaf $c_{equi}(X/\spec(k), 0)$ studied in 
\cite[Section 5.3]{Voev00-1}. Similarly for each $n$, the presheaf 
$U \mapsto \HH_{n}^{\Sus}(G;X \times U/U)$ is a homotopy invariant presheaf with equivariant 
transfers.
Therefore this defines a family of homotopy invariant equivariant pseudo pretheories.

\begin{lemma} 
\label{lem:sustr}
Let $F$ be a homotopy invariant equivariant pseudo pretheory on $\Sm^G_k$.
Let $S$ be an equivariantly irreducible smooth semilocal $G$-scheme and
$X/S$ be a relative smooth affine curve.
Let $D$ and $D'$ be effective equivariant Cartier divisors on $X$ which are finite and 
surjective over $S$. 
If the image of $(D-D')$ in $\HH_{0}^{\Sus}(G;X/S)$ vanishes, 
then $\tr_{D} = \tr_{D'}$.
Here $\tr_{D}$ and $\tr_{D'}$ denote the transfer maps associated to $D$ and $D'$, 
respectively.
\end{lemma}
\begin{proof}
The proof follows as in \cite[Lemma 6.3]{HVO}.
\end{proof}
\end{example}

\begin{example}
\label{ex:K_0^G} {\bf $K_0^G$-presheaves.}
The notion of $K_0$-presheaves was introduced and studied by Walker in
\cite{Wal96} (see also \cite[Section~1]{Sus03}). Homotopy invariant $K_0$-presheaves satisfy many
properties enjoyed by presheaves with transfers. An equivariant generalisation of this
notion was developed in \cite[Section~6.2]{HKO15}. We briefly recall the definition here.

For $X, Y \in \Sch^G_k$, let $\mcal{P}^G(X,Y)$ denote the category of coherent $G$-modules 
on $X \times Y$ which are flat over $X$ and whose support is finite over $X$.
This is an exact subcategory of the abelian category of coherent $G$-modules on
$X \times Y$. Define $K_0^G(X, Y) := K_0(\mcal{P}^G(X,Y))$. 
Given $X, Y, Z \in \Sm^G_k$, we have a natural biexact bifunctor
$\mcal{P}^G(X, Y) \times \mcal{P}^G(Y, Z) \to \mcal{P}^G(X,Z)$ given by 
$(P, Q) \mapsto 
(p_{XZ})_*(p_{XY}^*(P) \otimes p_{YZ}^*(Q))$,
where the tensor product is taken over $\mcal{O}_{X \times Y \times Z}$.
Thus we get a natural composition pairing of exact categories
$\circ: K_0^G(X,Y) \times K_0^G(Y,Z) \to K_0^G(X,Z)$ and
all these composition laws are associative. This allows us to define an
additive category $K_0(\Sm^G_k)$ by taking the objects of $\Sm^G_k$ to be the objects
and defining $\Hom_{K_0(\Sm^G_k)}(X,Y)= K_0^G(X,Y)$. 
A {\sl $K_0^G$-presheaf} is an additive presheaf of abelian groups 
on the category $K_0(\Sm^G_k)$.
Equivariant algebraic $K$-theory $K^G_i(-)$ is a $K_0^G$-presheaf for all $i$;
therefore, \aref{ex:K^G} is a special case of this one.

There is a functor $\Sm^G_k \to K_0(\Sm^G_k)$ which is the identity on objects and sends a 
morphism $q: X \to Y$ to the structure sheaf $\mcal{O}_{\Gamma_q}$ of the graph 
$\Gamma_q \subseteq X \times Y$.
In particular, a $K_0^G$-presheaf is also a presheaf on $\Sm^G_k$ and
we discuss below that it is in fact an equivariant pseudo pretheory.

Given an equivariant relative smooth affine curve $p: X \to S$ and an effective equivariant
Cartier divisor $i_D: D \hookrightarrow X$ which is finite and surjective over $S$,
the map $p_D:= p|_D: D \to S$ is a finite and flat equivariant map.
Let $\Gamma_{p_D}^t \subseteq  S \times D$ denote the transpose of the graph
of $p_D$ and let $\mcal{O}_{\Gamma_{p_D}^t}$ denote its structure sheaf. 
Then $\mcal{F}^t_D := (Id_S \times i_D)_* (\mcal{O}_{\Gamma_{p_D}^t}) 
\in \mcal{P}^G(S, X)$.
Define $\tr_D: F(X) \to F(S)$ to be 
$F(\mcal{F}^t_D)$.
Then the transfer maps $\tr_D$ are clearly compatible with pullbacks and sections.
If $D$ and $D'$ are as in \aref{lem:K^G-pseudo-pre}, then the
exact sequence \eqref{eqn:EPP1} gives an
exact sequence of coherent sheaves in $\mcal{P}^G(S,X)$:
$$
0 \to \mcal{F}^t_{D'}  \to \mcal{F}^t_{D+D'} \to \mcal{F}^t_{D}  \to 0.
$$
Using the additivity in $K_0^G(S,X)$,
it follows that $\tr_{D + D'} = \tr_{D} + \tr_{D'}$.
\end{example}

\begin{example} {\bf Bredon motivic cohomology.}
\label{ex:BredonMotivicCohomology}
Bredon motivic cohomology introduced in \cite[Section~5]{HVO} and further studied in \cite{HVO16}
(for smooth varieties equipped with $\Z/2\Z$-action) is an equivariant generalization of 
motivic cohomology for finite group actions.

For a smooth $G$-scheme $X$ over $k$,
recall that $\Z_{tr,G}(X)$ denotes the presheaf with equivariant transfers given by
$\Z_{tr,G}(X)(-) := \Cor_k^G(-,X)$. 
If $F$ is a presheaf of abelian groups on $\Sm^G_k$,
write $C^*F(X)$ for the cochain complex associated to the 
simplicial abelian group $F(X \times \Delta_{\bullet})$.
For a finite dimensional representation $V$ of
$G$, let $\Z_G(V)$ denote the complex of presheaves with equivariant transfers
given by: 
$$
\Z_G(V) := C^*(\Z_{tr,G}(\P(V \oplus 1))/\Z_{tr,G}(\P(V)))[-2 \dim(V)].
$$
The {\sl Bredon motivic cohomology} of a smooth $G$-variety $X$
is defined to be the equivariant Nisnevich hypercohomology with coefficients in $\Z_G(V)$: 
$$H^n_G(X,\Z(V)) := H^n_{GNis}(X,\Z_G(V)).$$ (See  \aref{sub:EqNis}
for the definition of the equivariant Nisnevich site.)

The fact that Bredon motivic cohomology define presheaves with
equivariant transfers follows from \cite[ Proposition 3.1.9]{Voev00} in the case of a trivial group
and is proved in \cite[Corollary 3.8]{HVO16} for $\Z/2\Z$.
The case of finite groups follows verbatim from the fact that 
smooth $G$-schemes have finite equivariant Nisnevich cohomological
dimension \cite[Corollary 3.9]{HVO} and \cite[Theorem 4.15(3)]{HVO}.
Therefore
Bredon motivic cohomology define equivariant pseudo pretheories.
\end{example}

\section{Equivariant Nisnevich topology and compactifications} 
\label{sec:GComp}
In this section we discuss the notions of equivariant Nisnevich topology and equivariant good compactification of equivariant smooth relative curves.
We establish some of their properties which are needed in the proofs of our rigidity theorems.

\subsection{Equivariant Nisnevich topology} We recall briefly the equivariant Nisnevich topology on $\Sm^G_k$ for finite groups, first introduced by Voevodsky in   \cite[Section~3.1]{Del09}.

\label{sub:EqNis}

\begin{definition} \label{def:Nis}
A {\sl distinguished square} in $\Sch^G_k$ is a cartesian square
\begin{equation} \label{eqn:N-1}
\xymatrix{
B \ar[d] \ar[r]
& Y \ar[d]^{p} \\
A \ar@{^{(}->}[r]^{j}
& X,  \\
}
\end{equation}
where $j$ is an equivariant open immersion, $p$  an equivariant \'etale morphism, and the induced map
$(Y \smallsetminus B)_{\red} \to (X \smallsetminus A)_{\red}$ is an 
isomorphism. The collection of distinguished squares 
forms a $cd$-structure in the sense of \cite[Definition~2.1]{Voev10}.
The associated Grothendieck topology is called the equivariant 
{\sl Nisnevich} topology. We write $(\Sm^G_k)_{GNis}$ (resp. $(\Sch^G_k)_{GNis}$) for the 
respective sites of smooth $G$-schemes and $G$-schemes equipped with the 
{\sl equivariant Nisnevich topology}.
\end{definition}

Equivariant Nisnevich covers admit the following equivalent characterizations (see \cite[Propositions 2.15, 2.17]{HKO15}).

\begin{proposition}
\label{prop:NisChar}
Let $f:Y\to X$ be an equivariant \'etale map between $G$-schemes. 
The following are equivalent.
\begin{enumerate}
\item 
The map $f$ is an equivariant Nisnevich cover.
\item 
There exists a sequence of invariant closed subschemes
$$
\emptyset = Z_{m+1}\subseteq Z_{m}\subseteq \cdots \subseteq Z_{1}\subseteq Z_{0}=X
$$
such that $f|_{f^{-1}(Z_{i}-Z_{i+1})}:f^{-1}(Z_{i}-Z_{i+1})\to Z_{i}-Z_{i+1}$ has an equivariant section.
\item 
For every $x\in X$, 
there exists a point $y\in Y$ such that $f$ induces isomorphisms of residue fields $k(x) \iso k(y)$ and set-theoretic stabilizers $G_{y} \iso G_{x}$.
\end{enumerate}
\end{proposition}

Let $X \in \Sch^G_k$ and suppose $x \in X$ has an invariant open affine neighborhood. 
Then the semilocal ring $\mcal{O}_{X,Gx}$ has a natural $G$-action which induces a $G$-action on the Henselian semilocal ring $\mcal{O}^h_{X,Gx}$ with a single closed orbit.
Any semilocal Henselian affine $G$-scheme over $k$ with a single orbit is equivariantly isomorphic to $\spec(\mcal{O}^h_{Y,Gy})$ for some affine $G$-scheme $Y$ and $y \in Y$.

For $X \in \Sch^G_k$ and any $x \in X$,
let $N_{G}(Gx)$ denote the filtering category of equivariant \'etale neighborhoods of $Gx$. Its objects are pairs $(p:U\to X, s)$, 
where $U$ is an equivariantly irreducible $G$-scheme, 
$p$ is an equivariant \'etale map, 
and $s:Gx\to U$ is an equivariant section of $p$ over $Gx$. 
A morphism from $(U\to X,s)$ to $(V\to X, s')$ in $N_{G}(Gx)$ is a map $f:U\to V$ making the evident triangles commute.
Although $x \in X$ might not be contained in any $G$-invariant affine neighborhood, 
it makes sense to consider $G \times^{G_x} \spec(\mcal{O}^h_{X,x})$ and according to \cite[Proposition 3.13]{HVO} we have:
\begin{equation} 
\label{eqn:Nispts1}
\lim_{U\in N_{G}(Gx)}U \iso \spec(\mcal{O}_{G\times^{G_{x}}X,Gx}^{h})
\iso 
G\times^{G_{x}}\spec(\mcal{O}_{X,x}^{h}).
\end{equation}
Further if $x\in X$ has an invariant affine neighborhood then there is a canonical $G$-isomorphism 
\begin{equation} \label{eqn:Nispts2}
G\times^{G_{x}}\spec(\mcal{O}^{h}_{X,x}) 
\xrightarrow{\iso} 
\spec(\mcal{O}^{h}_{X,Gx}).
\end{equation}

For a Nisnevich sheaf $F$ on $\Sm^G_k$, $X \in \Sm^G_k$, and $x \in X$, we set 
$$
p_x^*F := F(\spec(\mcal{O}^{h}_{G\times^{G_{x}}X,Gx})) = \colim_{U\in N_{G}(Gx)}F(U).
$$ 
Then $p_x^*$ defines a fiber functor from the category of sheaves to sets, 
i.e., 
it commutes with colimits and finite products and so determines a point of the $G$-equivariant Nisnevich topos. 
It is known that the set of points $\{p_x^* | x \in X, X \in \Sm^G_k \}$ forms a conservative set of points for $(\Sm^G_k)_{GNis}$ (see \cite[Theorem 3.14]{HVO}).

\subsection{Suslin homology of equivariant curves}

An equivariant map $p:X\to S$ is an equivariant curve if all of its fibers have dimension one.

\begin{definition}
\label{def:good}
Say that a smooth equivariant curve $p:X \to S$ admits a \textit{good  compactification} if $p$ factors as
\[
\xymatrix{
X \ar@{^{(}->}^{j}[r] \ar[dr]_-{p} & \overline{X} \ar[d]^-{\overline{p}} \\
& S,
}
\]
where $\overline{X}$ is normal, $\overline{p}$ is a proper equivariant curve, $j$ is an equivariant open embedding,  and $X_{\infty}=(\overline{X}\smallsetminus X)_{\red}$ has an invariant open affine neighborhood in $\overline{X}$. 
\end{definition}

The following lemma about base change is straightforward to verify. 

\begin{lemma}
\label{lem:bcofgdcpt}
Let $X \to S$ be an equivariant smooth curve and $S' \to S$ be an equivariant map, 
where $S, S'$ are affine $G$-schemes (smooth or a local or semilocal $G$-scheme which is a limit of smooth $G$-schemes).
If $X \to S$ admits an equivariant good compactification, 
then the smooth equivariant curve $X' = X \times_S S' \to S'$ also admits an equivariant good compactification.
\end{lemma}

If $S$ is affine and $X \to S$ is an equivariant smooth quasi-affine curve with equivariant good compactification $\overline{X}$ and $X_{\infty}=(\overline{X}\smallsetminus X)_{\red}$,
then the equivariant Suslin homology of $X/S$ can be interpreted in terms of relative equivariant Cartier divisors (see \cite[Theorem 3.1]{SV96} when $G$ is trivial, and \cite[Theorem 6.12]{HVO} for an extension to the equivariant case):
\begin{equation} 
\label{eqn:SusHom}
\HH^{\Sus}_{n}(G;X/S) 
\iso 
\begin{cases}
\Div^{G}_{rat}(\overline{X},X_{\infty})    & n=0 \\ 
0 & n>0.
\end{cases}
\end{equation}

\begin{lemma}
	Let $S= \lim_{\alpha\in A} S_{\alpha}$ be a cofiltered limit where the $S_{\alpha}$ are quasi-projective  $G$-schemes over $k$ and the transition maps are equivariant and affine. If $f:X\to S$ is a finite type equivariant map, then there is $\lambda$, a finite type $G$-scheme $X_{\lambda}$ over $k$, and an equivariant map $f_{\lambda}:X_\lambda\to S_{\lambda}$ fitting into a Cartesian square 
	\[
	\xymatrix{
		X \ar[r]\ar[d]_{f} & X_{\lambda} \ar[d]^{f_{\lambda}} \\
		S \ar[r] & S_{\lambda}.
	}
	\]
 Moreover if $f$ is satisfies any of the properties: (i) affine, (ii) open, (iii) smooth, (iv) proper,    then $f_\lambda$ can be chosen to have the same properties.  
\end{lemma}
\begin{proof}
Let $T_{\alpha} = S_{\alpha}/G$ and $T = \lim_{\alpha} T_{\alpha}$.
By \cite[Th\'eor\`eme 8.8.2]{EGAIV3}
there is $\beta$ and a map of finite type $T_{\beta}$-schemes $f_{\beta}:X_{\beta} \to S_{\beta}$ such that 	
$X\iso X_{\beta}\times_{S_{\beta}}S$ and under this isomorphism $f$ is the pullback of $f_\beta$. Moreover  if $f$ satisfies some of the properties (i)-(iv), then $f_{\beta}$ can be chosen to satisfy the same properties 
\cite[Th\'eor\`eme 8.10.5]{EGAIV3}, \cite[Proposition 17.7.8]{EGAIV4}. 
For $\alpha\geq \beta$, set $X_\alpha = X_{\beta}\times_{S_\beta}S_{\alpha}$.
We have that
$\Aut_{T}(X) \iso \colim_{\alpha} \Aut_{T_{\alpha}}(X_{\alpha})$. Since $G$ is finite, the homomorphism $G\to\Aut_T(X)$ factors through some $\Aut_{T_{\lambda}}(X_{\lambda})$, i.e., we may choose $X_{\lambda}$ to have a $G$-action.    Increasing $\lambda$ we can further assume that $f_\lambda$ is equivariant.
\end{proof}

\begin{lemma} \label{lem:Sus-Hom}
Let	 $S = \lim_{\alpha\in A}S_{\alpha}$ be a cofiltered limit where $S_\alpha\in \Sm^G_k$ are affine and the transition maps are equivariant \'etale. Let $X\to S$ be a smooth equivariant affine curve admitting good compactification. 
	\begin{enumerate}
		\item $\HH^{\Sus}_{n}(G;X/S) \iso \colim_{\beta} \HH^{\Sus}_{n}(G;X_{\beta}/S_{\beta})$ where $X_{\beta}\to S_{\beta}$ are smooth equivariant curves with good compactification. 
		\item $\HH^{\Sus}_{0}(G;X/S) \iso 
		\Div^{G}_{rat}(\overline{X}, X_{\infty})$ and $\HH^{\Sus}_{i}(G;X/S) = 0$ for $i>0$.
	\end{enumerate}
\end{lemma}
\begin{proof}
Let $X\subseteq \overline{X}$ be an equivariant good compactification. 
By the previous lemma, there is a smooth, affine, equivariant map  $X_{\alpha}\to S_{\alpha}$, with equivariant compactification $\overline{X}_{\alpha}\to S_{\alpha}$ with $\overline{X}_{\alpha}\smallsetminus X_\alpha$ has an affine neighborhood, such that $X \iso X_\alpha\times_{S_\alpha}S$ and $\overline{X} \iso \overline{X}_\alpha\times_{S_\alpha}S$. 
For any generic point $\eta'\in X_{\alpha}$ lying over a generic point $\eta\in S_\alpha$, we have $\dim(\mcal{O}_{X_\alpha,\eta'}) = \dim(\mcal{O}_{S_\alpha,\eta}) +1$. Thus there
 is an open subset of $U\subseteq S_\alpha$ over which the fibers of $X_\alpha$, $\overline{X}_{\alpha}$ are one dimensional. Since $U$ contains the image of $S$ in $S_\alpha$, there is $\lambda\geq\alpha$ such that $X_\lambda$ and $\overline{X}_{\lambda}$ are equivariant  curves over $S_\lambda$, where $X_{\beta} = X_{\alpha}\times_{S_{\alpha}}S_{\beta}$ for $\beta\geq \alpha$ and similarly for $\overline{X}_{\beta}$. Replacing $\overline{X}_{\lambda}$ by its normalization, we see that $X_{\lambda}\to S_{\lambda}$ admits good compactification. We thus have that $X\to S$ is isomorphic to the cofiltered limit
$\lim_{\beta\geq \lambda}(X_{\beta}\to S_{\beta})$ of smooth affine equivariant curves admitting good compactification. Moreover, we have 
$\colim_{\beta}C_{n}(X_{\beta}/S_{\beta})\iso C_{n}(X/S)$ and taking fixed points and homology commutes with filtered colimits, yielding (1).

Write $X\to S$ as a filtered limit 
$\lim_{\beta\in B}(X_{\beta}\to S_{\beta})$ of equivariant curves with good compactification. Moreover we can assume $B$ has a minimal element $0$ and $\overline{X}_{\beta} = \overline{X}_{0}\times_{S_{0}}S_{\beta}$ is a good compactification of $X_{\beta}$. Write $Y_{\beta} = \overline{X}_{\beta}\smallsetminus X_{\beta}$. 
 Under the isomorphism \eqref{eqn:SusHom}, the map $\HH^{\Sus}_{0}(G;X_{\beta}/S_{\beta}) \to \HH_0^{\Sus}(G;X_{\alpha}/S_{\alpha})$ agrees with the map 
$\Div^{G}_{rat}(\overline{X}_{\beta},Y_{\beta}) \to \Div^{G}_{rat}(\overline{X}_{\alpha},Y_{\alpha})$ and so $\HH^{\Sus}_{n}(G;X/S) \iso \colim_{\beta} \Div^{G}_{rat}(\overline{X}_{\beta},Y_{\beta})$. Finally, note that 
$\colim_{\beta} \Div^{G}_{rat}(\overline{X}_{\beta},Y_{\beta})\iso  \Div^{G}_{rat}(\overline{X},X_{\infty})$.
\end{proof}

\begin{corollary}
Let $F$ be a homotopy invariant equivariant pseudo pretheory on $\Sm^G_k$ and 
$X\to S$ as in the statement of the previous lemma. Then there is a pairing of abelian groups
\[
\HH^{\Sus}_{0}(G;X/S)\otimes F(X) \to F(S).
\]
\end{corollary}

\begin{proposition}\label{prop:Hsusinj}
	Let $S= \spec(\mcal{O}^{h}_{W,Gw})$ be the Henselization of a smooth affine $G$-scheme $W$ at an orbit $Gw$. 
	Let $p:X\to S$ be a smooth equivariant affine curve with an equivariant good compactification.
	Let $X_{0}\to S_{0}$ be the fiber over the closed orbit $S_{0}$ in $S$. 
	Then for any $n$ coprime to $\ch(k)$, 
	restriction induces an injection
	$$
	\HH_{0}^{\Sus}(G;X/S)/n \hookrightarrow \HH^{\Sus}_{0}(G;X_{0}/S_{0})/n.
	$$
\end{proposition}
\begin{proof}
Let $\overline{X}$ be the equivariant good compactification of $X$ over $S$ such that 
$Y =(\overline{X}\smallsetminus X)_{\red}$ has an invariant open neighborhood in 
$\overline{X}$.
By Lemma~\ref{lem:Sus-Hom}(2) and \cite[Proposition~6.8]{HVO} it suffices to show that the restriction 
$\Pic^{G}(\overline{X}, Y)/n \to \Pic^{G}(\overline{X}_{0}, Y_{0})/n$ is injective. 
	This follows as in the proof of \cite[Theorem 4.3]{SV96}, 
	by replacing \'etale cohomology with $H^{*}_{et}(G;-)$ and classical proper base change with 
	\aref{thm:basechange}.
\end{proof}

\section{Rigidity for equivariant pseudo pretheories}
\label{sec:rigidity}
In this section we establish versions of the rigidity theorems of Suslin \cite{Sus83}, Gabber \cite{Gab92}, and Gillet and Thomason \cite{GT84} in the setting of equivariant pseudo pretheories.

\begin{theorem}[Equivariant Suslin Rigidity]\label{thm:suslinrigid}
Let $F$ be a homotopy invariant equivariant pseudo pretheory on $\Sm^G_k$ which takes values in torsion abelian groups 
of exponent coprime to $\ch(k)$. 
Let $S= \spec(\mcal{O}^{h}_{W,Gw})$ be the Henselization of a smooth affine $G$-scheme 
$W$ 
at a closed orbit, 
and $X\to S$ a smooth affine equivariant curve admitting good compactification.  
If $i_{1}, i_{2}:S\to X$ are two equivariant sections which coincide on the closed orbit of $S$,
then $i_{1}^{*} = i_{2}^{*}: F(X) \to F(S)$. 
\end{theorem}
\begin{proof}
For any $n$, $F_{n} = \ker(n:F\to F)$ is again a homotopy invariant equivariant pseudo 
pretheory and $F = \cup_{n} F_{n}$. 
Thus it suffices to consider the case when $nF = 0$. 
We may assume that $X$ is equivariantly irreducible.
The images of the sections $i_{j}$ are closed subschemes $W_{j}\subseteq X$ which are 
elements of $C_{0}(X/S)^{G}$. 
By definition we have $i_{j}^{*} = \tr_{W_{j}}$. 
By  \aref{lem:sustr} it suffices to show that $W_{1}-W_{2}$ 
becomes zero in $\HH_{0}^{\Sus}(G;X/S)/n$. 
 \aref{prop:Hsusinj} shows that there is an injection $\HH_{0}^{\Sus}(G;X/S)/n 
\hookrightarrow \HH_{0}^{\Sus}(G;X_{0}/S_{0})/n$, 
where $X_{0}$ is the fiber over the closed orbit $S_{0}$ of $S$.
Since $i_1$ and $i_{2}$ coincide on the closed orbit, 
we conclude that $W_{1}-W_{2}$ is zero in $\HH_{0}^{\Sus}(G;X/S)/n$. 

\end{proof}

Recall that we write $R\wr G$ for the skew group ring. 

\begin{lemma}\label{lem:etnbd}
Let  $X\to Z$ be a map in $\Sm^G_k$, with $X$ affine, $Z=\spec(L)$ where $L$ is a field, and $x\in X$ an invariant closed point such that $k(x)\iso L$. 
Then there is a commutative diagram in $\Sm^G_k$
\[
\xymatrix@C-8pt{
X \ar[rr]^{\phi}\ar[dr]& & \mathcal{V} \ar[dl] \\
& Z, &
}
\]
 where $\mathcal{V}$ is an equivariant vector bundle over $Z$, $\phi$ is \'etale at $x$, and $\phi(x) = 0$. 
 
\end{lemma}
\begin{proof}
	Write $X= \spec(A)$ and $m\subseteq A$ for the maximal ideal corresponding to $x$. Since $|G|$ is invertible in $L$, the surjection of $L\wr G$-modules $m\to m/m^2$ has a splitting. The resulting map of $L\wr G$-modules $m/m^2 \to m \subseteq A$ induces the equivariant ring map $\Sym(m/m^2) \to A$. Applying $\spec$ yields the desired map.  
\end{proof}

\begin{lemma} \label{lem:useinrig}
Let $x\in X$ be an invariant closed point, $X\to \spec(L)$, and $\VV$ be as in the previous lemma. Assume that there is an equivariant vector bundle isomorphism $\VV\iso \WW\oplus \VV'$, where $\WW$ has rank $\dim(X) -1$, and let $p:X\to \WW$ be the resulting map. 
Then there are invariant open affine neighborhoods $U\subseteq X$ and $S\subseteq \WW$ of $x$ and $0$ respectively, such that $p$ induces a smooth equivariant curve $U\to S$ admitting good compactification. 
\end{lemma}
\begin{proof}
First consider the case where $X\subseteq \VV$ is an invariant open subscheme with closure $\overline{X} = \WW\times \P(\VV' \oplus \mcal{O}_L)$.
For any $a\in X$, the fiber of $X_{p(a)}$ has dimension one and so  $(\overline{X}\setminus X)_{p(a)}$ must be finite over $p(a)$ (where $\overline{X}\setminus X$ is considered as a closed subscheme with reduced structure). 
Since $\overline{X}$ is projective over an affine scheme, there is an invariant affine neighborhood $A\subseteq \overline{X}$ of the finite set of closed points $(\overline{X}\setminus X)_{0}$. 
Then $Z= (\overline{X}\smallsetminus X) \smallsetminus ((\overline{X}\smallsetminus X) \cap A)$ is closed in $\overline{X}$ and so has closed image in $\WW$.  
Now let $S\subseteq \WW$ be an invariant affine neighborhood of $ 0$ which misses the image of $Z$ and is contained in $p(X)$ (we can find an affine neighborhood with these properties and the intersection over all the translates by $g\in G$ is an invariant neighborhood). 
Now let $U = X_{S}$ and $U' = \overline{X}_S$. Then  $U'\smallsetminus U$ has an invariant affine neighborhood. 
 Let $\overline{U}$ be the normalization of $U'$. Then $\overline{U}$ inherits a $G$-action from that on $U'$ and contains $U$ as an invariant open subscheme. Since $\overline{U}\to U'$ is finite, $\overline{U}\smallsetminus U$ is contained in an invariant affine neighborhood.  
Now $U\to S$ is a smooth equivariant curve with good compactification $\overline{U}$.

In the general case, since $\phi:X\to \VV$ is \'etale at $x$, there is an open invariant affine neighborhood on which $\phi$ is \'etale, so shrinking $X$, we may assume $\phi$ is \'etale.  By the previous paragraph, there are invariant affine neighborhoods $M\subseteq \phi(X)$ of $0$ and $S\subseteq \WW$  such that 
$M\to S$ is an equivariant smooth affine curve with good compactification $\overline{M}$. Then $U:=\phi^{-1}(M)\to \overline{M}$ 
 is equivariant and quasi-finite and so
the equivariant version of Zariski's main theorem (see \cite[Theorem~16.5]{LMB}) yields an 
equivariant factorization of $U \to \overline{M}$ as the composition of an invariant open immersion 
$U \hookrightarrow \overline{U}$ and an equivariant finite map $q: \overline{U} \to \overline{M}$. 
Replacing $\overline{U}$ by its normalization, 
we may assume $\overline{U}$ is normal.
Since $\overline{M}$ is an equivariant good compactification of $M$ over $S$ and $q$ is affine, 
it follows that $\overline{U}$ is an equivariant good compactification of $U$ over $S$.

\end{proof}

\begin{theorem}[Equivariant Gabber Rigidity]
\label{thm:gabberrigid}
Assume that every $G$-representation over $k$ is a direct sum of one dimensional representations.
Let $F$ be a homotopy invariant equivariant pseudo pretheory on $\Sm^G_k$ with torsion values of exponent coprime to $\ch(k)$. 
If $X$ is a smooth affine $G$-scheme over $k$ of pure dimension $d$ and $x\in X$ is a closed point such that $k\subseteq k(x)$ is separable,
then there is an isomorphism:
$$
F(Gx) \xrightarrow{\iso} F(\spec(\mcal{O}_{X,Gx}^{h})).
$$ 
\end{theorem}
\begin{proof}
We proceed by induction on $d=\dim(X)$, the case $d=0$ being clear.
By \eqref{eqn:Nispts2}, there is an equivariant isomorphism
$$
G\times^{G_{x}} \spec(\mcal{O}_{X,x}^{h}) {\overset{\iso} \to} \spec(\mcal{O}_{X,Gx}^{h}).
$$ 
Thus we are reduced to showing there is an isomorphism 
$$
\epsilon^{*}F(\spec(k(x))) {\overset{\iso} \to} \epsilon^{*}F(\spec(\mcal{O}^{h}_{X,x})),
$$ 
where $\epsilon(-) = G\times^{G_{x}} (-)$ and $\epsilon^{*}F:=F\circ \epsilon$.  
Note that $\epsilon^{*}F$ is a homotopy invariant equivariant pseudo pretheory on $\Sm^{G_x}_k$ which is torsion of exponent coprime to $\ch(k)$. 
Replacing $G$ by $G_{x}$ and $F$ by $\epsilon^{*}F$ it suffices to consider the case where $Gx$ consists of a single point.

The projection $X_{x}\to X$ sends equivariant \'etale neighborhoods of $x\in X_{x}$ to equivariant \'etale neighborhoods of $x\in X$. If 
$U\to X$ is an equivariant \'etale neighborhood of $x\in X$, then $U_{x}\to X$ is an equivariant \'etale neighborhood of  
$x\in X$ mapping to $U$. This implies that  
$\spec(\mcal{O}^h_{X_{x},x}) \iso \spec(\mcal{O}^h_{X,x})$ and so we may replace $X$ with $X_{x}$ and assume there is an equivariant map $X\to \spec(L)$, where $L=k(x)$ (equipped with the corresponding $G$-action).    Furthermore, by \aref{lem:etnbd} there is an equivariant vector bundle $\mathcal{V}$ over $\spec(L)$ such that  
$\mcal{O}^{h}_{X,x} \iso \mcal{O}^{h}_{\mathcal{V},0}$ 
and so it suffices to assume 
$X = \mathcal{V}$ and $x=0_L\in \mathcal{V}$.

The assumption on $G$ implies that there is a representation $V'$ over $k$ and
an equivariant isomorphism $\mathcal{V} \iso \A(V')_{L}$, see e.g., 
the beginning of the proof of \cite[Theorem 8.11]{HVO}. In particular, $\VV$ is a direct sum of equivariant line bundles. Let $i:\WW\subseteq \VV$ be a rank $d-1$ summand. It now suffices to see that $i^*$ induces an isomorphism 
$F(\spec(\mcal{O}^h_{\VV,0}))\iso F(\spec(\mcal{O}^h_{\WW,0}))$, since  $0_L\in \mathcal{W}$ and the induction hypothesis implies that $F(\mathcal{W})\iso F(0_L)$. The inclusion $i$ is split by the projection $p:\VV \to \WW$, so it suffices to see that $i^*$ is injective.

Suppose that $[\alpha]\in F(\spec(\mcal{O}^h_{\VV,0}))$ is such that $i^*([\alpha])= 0$. By definition $F(\spec(\mcal{O}^h_{\VV,0})) = \colim_{U \to \VV} F(U)$, where the colimit is over equivariant \'etale neighborhoods of $0_L\in \VV$. Thus, there is a representative $\alpha \in F(U)$ of $[\alpha]$ where
$U \to \VV$  is
an affine equivariant \'etale neighborhood
  of $0_L$. There is a canonical equivariant map $\pi:\spec(\mcal{O}^h_{\VV,0}) \to U$.

After shrinking $U$, there is a smooth affine equivariant curve $U\to Y$, admitting a good compactification, 
by \aref{lem:useinrig},  where  $Y\subseteq \WW$ is an invariant neighborhood of $0$. 
Consider the following commutative diagram of equivariant maps:
\[
\xymatrix@C-6pt{
 \spec(\mcal{O}^h_{\VV,0}) \ar@/^/[drrr]^{s_i} \ar[dr]^{j_i} \ar@/_/[ddr]_{{\rm id}}& & &\\
 & \tilde{U} \ar[rr]^{q_2} \ar[d]^{q_1} &  & U \ar[d]\\ 
 & \spec(\mcal{O}^h_{\VV,0}) \ar[r]^-{p} & \spec(\mcal{O}^h_{\WW,0}) \ar[r] & Y,\\
}
\]
where the rectangle is a pullback.   By \aref{lem:bcofgdcpt}, $\tilde{U}\to \spec(\mcal{O}^h_{\VV,0})$
is a smooth affine equivariant curve admitting good compactification.
The maps $s_1:=\pi$ and  $s_2 := \pi \circ i \circ p$ induce 
equivariant sections $j_1, j_2 : \spec(\mcal{O}^h_{\VV,0})  \to \tilde{U}$ of $q_1$.
The sections $j_1, j_2$ agree on the closed orbit by construction and 
therefore $j_1^* = j_2^*$ by \aref{thm:suslinrigid}. Thus $[\alpha] = \pi^*\alpha = p^*i^*\pi^*\alpha = 0$.

\end{proof}

\section{On the equivariant Gersten resolution} 
\label{sec:Gerst-resol}
For an affine $G$-scheme $X \in \Sch^G_k$, 
let $\mcal{M}^G(X)$ denote the abelian category of $G$-equivariant coherent $\mcal{O}_X$-modules.
For $p \geq 0$, 
let $\mcal{M}^{G,p}(X) \subset \mcal{M}^G(X)$ denote the Serre subcategory of coherent sheaves $\mcal{F}$ whose support is a subscheme of codimension $\geq p$ in $X$.
Since $\mcal{F}$ is equivariant, 
the support is an invariant closed subscheme of $X$.
Let $S^{G,p}(X)$ denote the set of all distinct set-theoretic $G$-orbits $[x]$ in $X$ 
of codimension $p$ points $x$ of $X$. 
Consider the filtration of $\mcal{M}^G(X)$ by Serre subcategories
$$
\mcal{M}^G(X) = \mcal{M}^{G,0}(X) \supset \mcal{M}^{G,1}(X) \supset \mcal{M}^{G,2}(X) \supset \cdots \supset
\mcal{M}^{G,p}(X) \cdots.
$$
Since the natural exact functor 
$\mcal{M}^{G,p}(X) \to {\underset{[x] \in S^{G,p}(X)} \coprod} {\underset{n} \bigcup}  \mcal{M}^G 
(\spec( \mcal{O}_{X, Gx}/ J_{Gx}^n))$ has kernel $\mcal{M}^{G,p+1}(X)$ and
admits a section functor, by \cite[Proposition III.2.5]{Gab62} 
we have an equivalence of categories:
$$
\frac{\mcal{M}^{G,p}(X)}{\mcal{M}^{G,p+1}(X)} \xrightarrow{\simeq} 
{\underset{[x] \in S^{G,p}(X)} \coprod} {\underset{n} \bigcup}  \mcal{M}^G 
(\spec( \mcal{O}_{X, Gx}/ J_{Gx}^n)),
$$
where $J_{Gx}$ denotes the Jacobson radical of the semilocal ring $\mcal{O}_{X, Gx}$.
The Devissage theorem \cite[Theorem 4]{Qui73}, the Chinese remainder theorem 
and the equivalence of equivariant $K$-theory and $G$-theory for regular $G$-schemes 
\cite[Theorem 5.7]{T87}
imply that
$$ 
K_q^G({\underset{y \in [x]} \coprod} \spec(k(y))) \simeq 
G_q^G({\underset{y \in [x]} \coprod} \spec(k(y))) \simeq 
K_q(\mcal{M}^G (\spec( \mcal{O}_{X, Gx}/ J_{Gx}^n))),
$$
for every $n$. This yields an isomorphism 
for the union along all $n$.
Further for any $x \in X$, we have the Morita isomorphism \cite[Proposition 6.3]{T87}
$$
K_q^G({\underset{y \in [x]} \coprod} \spec(k(y))) \simeq K_q^{G_x}(\spec(k(x))).
$$
Combining the above and by \cite[Theorem 5]{Qui73}, 
for each $p \geq 0$ there is a localization sequence 
$$
\begin{array}{ll}
\cdots \rightarrow K_i(\mcal{M}^{G,p+1}(X)) \rightarrow K_i(\mcal{M}^{G,p}(X)) \rightarrow
& {\underset{[x] \in S^{G,p}(X)} \coprod} K_i^{G_x}(\spec(k(x))) \rightarrow \\
&  K_{i-1}(\mcal{M}^{G,p+1}(X)) \rightarrow \cdots.
\end{array}
$$

The above gives rise to a strongly convergent spectral sequence 
$$
E_1^{p,q} = {\underset{[x] \in S^{G,p}(X)} \coprod} K_{-p-q}^{G_x}(\spec(k(x))) 
\Rightarrow 
G^G_{-p-q}(X).
$$
For $X \in Sm^G_k$, the spectral sequence yields a sequence of abelian groups
\begin{equation} 
\label{eqn:SS1}
\begin{array}{ll}
0 \xrightarrow{} K_n^G(X) 
 \xrightarrow{} {\underset{[x] \in S^{G,0}(X)} \coprod} K_n^{G_x}(\spec(k(x)))  
& \xrightarrow{d_1} {\underset{[x] \in S^{G,1}(X)} \coprod} K_{n-1}^{G_x}(\spec(k(x))) 
\xrightarrow{d_1} \\
& {\underset{[x] \in S^{G,2}(X)} \coprod} K_{n-2}^{G_x}(\spec(k(x))) 
\xrightarrow{d_1} \cdots,
\end{array}
\end{equation}
where $d_1$ is the differential on the $E_1$-terms of the spectral sequence.

The Gersten conjecture states that \eqref{eqn:SS1} is exact if $G$ is trivial and $X = \spec(R)$, 
where $R$ is a regular local ring.
This is known for regular local rings containing a field, the geometric case was proved
by Quillen \cite[Theorem 5.11]{Qui73} and the general equicharacteristic case was proved
by Sherman \cite{Sh78} in the 1-dimensional case and Panin \cite{Pan03} for higher dimensions.
If $X$ is a regular local ring containing a field with a trivial $G$-action, where $G$ is a finite 
diagonalizable group, 
then the Gersten sequence \eqref{eqn:SS1} is simply the tensor product of the non-equivariant 
Gersten sequence with the group ring $\Z[G]$ (by \cite[Section 3.4]{Ser68}), and
is therefore exact. 
If the action of $G$ is non-trivial, 
we discuss in  \aref{ex:Gerst-resol} below that the sequence \eqref{eqn:SS1} need not be exact
even for $n=0$.

\begin{example} 
\label{ex:Gerst-resol}
Let $G = \Z/2\Z$ act on $X = \A^1_k = \spec(k[t])$ via the map $t \mapsto -t$.
For the closed point $x = (t) \in \A^1_k$ the Henselization $\mcal{O}^h_{X,x}$ is the ring of algebraic formal power series in $t$ over $k$.
We compute the $G$-equivariant $K_0$ with mod-$l$ coefficients of 
$\A^1_{(x)} := \spec (\mcal{O}_{X,Gx})$, 
$\spec (\mcal{O}^h_{X,Gx})$, 
the orbit $Gx$, 
and the generic point $\eta \in X$.

By \cite[Proposition~6.2]{T87} there is an isomorphism 
$$
K_0^G(Gx) \xrightarrow{\cong} K_0^{G_x}(\spec(k)),
$$
where the set-theoretic stabilizer $G_x$ of $x$ is equal to $G = \Z/2\Z$.
We have 
$$
K_0^G (\spec(k); l) 
\cong 
K_0^G(\spec(k)) \otimes \Z/l \Z 
\cong 
\Z/l\Z\oplus \Z/l\Z.
$$

Thus for a field $k$ of characteristic coprime to $2$, $l$,  \aref{thm:gabberrigid} implies
$$
K_0^G (\spec(\mcal{O}^h_{X,Gx}); l) 
\cong 
K_0^G (Gx; l) 
\cong 
\Z/l\Z \oplus \Z/l\Z.
$$ 

The natural map $\pi: \A^1_{(x)} \to \spec(k)$ affords a $G$-equivariant factorization:
\[
\xymatrix@R=.5em{
\A^1_{(x)}  \ar[rr]^-{\pi} \ar[dr]_-{j} & & \spec(k) \\
& \A^1_k \ar[ur]_-{\pi_1}.&\\}
\]
Here $j^*: K_0^G(\A^1_k) \to K_0^G(\A^1_{(x)})$ is surjective by the localization exact sequence, 
and $\pi_1^*: K_0^G(\spec(k)) \to K_0^G(\A^1_k)$ is an isomorphism \cite[Theorems 2.7, 5.7, 4.1]{T87}.
It follows that $\pi^*: K_0^G(\spec(k)) \to K_0^G(\A^1_{(x)})$ is surjective.
Since $\pi: \A^1_x \to \spec(k)$ has an equivariant section given by $t \mapsto 0$, 
$\pi^*: K_0^G(\spec(k)) \to K_0^G(\A^1_{(x)})$ is also injective.
Therefore $K_0^G(\A^1_{(x)}; l) \cong K_0^G(\spec(k); l) \cong \Z/l\Z \oplus \Z/l\Z$.
 
For the generic point $\eta = \spec(k(t))$, 
note that the $G$-action on $k(t)$ is free and $k(t)^G = k(t^2)$. 
Therefore,
$K_0^G(\eta; l) \cong K_0(k(t^2)) \otimes \Z/l\Z \cong \Z/l\Z$ so that 
$K_0^G(\A^1_{(x)}; l) \ncong K_0^G(\eta; l)$. 
\end{example}

\begin{remark} \label{rem:Gersten}
As pointed out by the referee, the Gersten complex for $\A^1_{(x)}$ with action of
the group $G = \Z/2\Z$ given by $t \mapsto -t$ as in the above example can be
analyzed using the localization sequence as follows. 
Under the notations of example \ref{ex:Gerst-resol}, we get an exact sequence:
$$
\cdots \to K_1^G(\spec(k(t)) \xrightarrow{\partial} K_0^G(\spec(k)) \xrightarrow{x_*}
K_0^G(\A^1_{(x)}) \xrightarrow{\eta^*} K_0^G(\spec(k(t))).
$$
Now the closed point $x \in \A^1_{(x)}$ can be seen as the zero set of the diagonal section
of the line bundle $L = \A^1_{(x)} \times \A^1_{k} \to \A^1_{(x)}$,
where $\A^1_{k}$ has the above non-trivial $G$-action. By a variant of the
excess intersection formula for equivariant $K$-theory \cite[Theorem 3.8]{Koc98},
$x_*(1) = 1 - [L]$, and this class is non-zero in $K_0^G(\A^1_{(x)})$. Thus 
$\eta^*$ is not injective. The above considerations give the geometric reason for this:
as soon as the top Chern class (in equivariant $K$-theory of the point) of the normal bundle
is non-trivial, then $x_*$ is non-zero and $\eta^*$ is not injective. In the cases
considered in other articles, the normal bundle has trivial action, so the top Chern class is zero
and the map $\eta^*$ is injective.
\end{remark}

The rigidity property and the exactness of the Gersten sequence \eqref{eqn:SS1} are two important properties of algebraic $K$-theory of semilocal rings.
In  \aref{ex:K^G} and  \aref{thm:gabberrigid}, 
we prove the rigidity theorem for equivariant $K$-theory of schemes with finite group actions.
 \aref{ex:Gerst-resol} (see also \cite[Section 5.3]{nguyen16}) shows that the Gersten sequence is not exact for equivariant $K$-theory of semilocal rings with non-trivial $\Z/2\Z$-actions.
In this respect the cases of trivial and non-trivial actions are very different.

\bibliography{EquiRig_July3}

\providecommand{\bysame}{\leavevmode\hbox to3em{\hrulefill}\thinspace}
\providecommand{\MR}{\relax\ifhmode\unskip\space\fi MR }
\providecommand{\MRhref}[2]{%
  \href{http://www.ams.org/mathscinet-getitem?mr=#1}{#2}
}
\providecommand{\href}[2]{#2}
\begin{thebibliography}{MFK94}

\bibitem[AD]{AD}
A.~Ananyevskiy and A.~Druzhinin, \emph{Rigidity for linear framed presheaves
  and generalized motivic cohomology theories}, arXiv:1704.03483.

\bibitem[Ayo14]{MR3205601}
J.~Ayoub, \emph{La r\'ealisation \'etale et les op\'erations de
  {G}rothendieck}, Ann. Sci. \'Ec. Norm. Sup\'er. (4) \textbf{47} (2014),
  no.~1, 1--145. \MR{3205601}

\bibitem[CD16]{MR3477640}
D.-C. Cisinski and F.~D\'eglise, \emph{{\'E}tale motives}, Compos. Math.
  \textbf{152} (2016), no.~3, 556--666. \MR{3477640}

\bibitem[CR62]{CR}
C.~Curtis and I.~Reiner, \emph{Representation theory of finite groups and
  associative algebras}, Pure and Applied Mathematics, Vol. XI, Interscience
  Publishers, a division of John Wiley \& Sons, New York-London, 1962.
  \MR{0144979 (26 \#2519)}

\bibitem[Del09]{Del09}
P.~Deligne, \emph{Voevodsky's lectures on motivic cohomology 2000/2001},
  Algebraic topology, Abel Symp., vol.~4, Springer, Berlin, 2009, pp.~355--409.
  \MR{2597743}

\bibitem[FS02]{FS02}
E.~M. Friedlander and A.~Suslin, \emph{The spectral sequence relating algebraic
  {$K$}-theory to motivic cohomology}, Ann. Sci. \'Ecole Norm. Sup. (4)
  \textbf{35} (2002), no.~6, 773--875. \MR{1949356}

\bibitem[Gab62]{Gab62}
P.~Gabriel, \emph{Des cat\'egories ab\'eliennes}, Bull. Soc. Math. France
  \textbf{90} (1962), 323--448. \MR{0232821}

\bibitem[Gab92]{Gab92}
O.~Gabber, \emph{{$K$}-theory of {H}enselian local rings and {H}enselian
  pairs}, Algebraic {$K$}-theory, commutative algebra, and algebraic geometry
  ({S}anta {M}argherita {L}igure, 1989), Contemp. Math., vol. 126, Amer. Math.
  Soc., Providence, RI, 1992, pp.~59--70. \MR{1156502}

\bibitem[Gab94]{Gab94}
\bysame, \emph{Affine analog of the proper base change theorem}, Israel J.
  Math. \textbf{87} (1994), no.~1-3, 325--335. \MR{1286833}

\bibitem[Gro66]{EGAIV3}
A.~Grothendieck, \emph{{\'E}l\'ements de g\'eom\'etrie alg\'ebrique. {IV}.
  \'etude locale des sch\'emas et des morphismes de sch\'emas. {III}}, Inst.
  Hautes \'Etudes Sci. Publ. Math. (1966), no.~28, 255. \MR{0217086}

\bibitem[Gro67]{EGAIV4}
\bysame, \emph{{\'E}l\'ements de g\'eom\'etrie alg\'ebrique. {IV}. \'etude
  locale des sch\'emas et des morphismes de sch\'emas {IV}}, Inst. Hautes
  \'Etudes Sci. Publ. Math. (1967), no.~32, 361. \MR{0238860}

\bibitem[GT84]{GT84}
H.~Gillet and R.~W. Thomason, \emph{The {$K$}-theory of strict {H}ensel local
  rings and a theorem of {S}uslin}, Proceedings of the {L}uminy conference on
  algebraic {$K$}-theory ({L}uminy, 1983), vol.~34, 1984, pp.~241--254.
  \MR{772059}

\bibitem[HK{\O}15]{HKO15}
J.~Heller, A.~Krishna, and P.~A. {\O}stv{\ae}r, \emph{Motivic homotopy theory
  of group scheme actions}, J. Topol. \textbf{8} (2015), no.~4, 1202--1236.
  \MR{3431674}

\bibitem[HV{\O}15]{HVO}
J.~Heller, M.~Voineagu, and P.~A. {\O}stv{\ae}r, \emph{Equivariant cycles and
  cancellation for motivic cohomology}, Doc. Math. \textbf{20} (2015),
  269--332. \MR{3398714}

\bibitem[HV{\O}16]{HVO16}
\bysame, \emph{Topological comparison theorems for {B}redon motivic
  cohomology}, arXiv preprint arXiv:1602.07500 (2016).

\bibitem[HY07]{HY07}
J.~Hornbostel and S.~Yagunov, \emph{Rigidity for {H}enselian local rings and
  {$\Bbb A^1$}-representable theories}, Math. Z. \textbf{255} (2007), no.~2,
  437--449. \MR{2262740}

\bibitem[Jan]{UJ}
U.~Jannsen, \emph{Rigidity theorems for {K}- and {H}-cohomology and other
  functors}, arXiv:1503.08742.

\bibitem[K{\"o}c98]{Koc98}
B.~K{\"o}ck, \emph{The {G}rothendieck-{R}iemann-{R}och theorem for group scheme
  actions}, Ann. Sci. \'Ecole Norm. Sup. (4) \textbf{31} (1998), no.~3,
  415--458. \MR{1621405}

\bibitem[Kri10]{Kri10}
A.~Krishna, \emph{Gersten conjecture for equivariant {$K$}-theory and
  applications}, Math. Ann. \textbf{347} (2010), no.~1, 123--133. \MR{2593287}

\bibitem[LMB00]{LMB}
G.~Laumon and L.~Moret-Bailly, \emph{Champs alg\'ebriques}, Ergebnisse der
  Mathematik und ihrer Grenzgebiete. 3. Folge. A Series of Modern Surveys in
  Mathematics [Results in Mathematics and Related Areas. 3rd Series. A Series
  of Modern Surveys in Mathematics], vol.~39, Springer-Verlag, Berlin, 2000.
  \MR{1771927}

\bibitem[LS08]{LS08}
M.~Levine and C.~Serp\'e, \emph{On a spectral sequence for equivariant
  {$K$}-theory}, $K$-Theory \textbf{38} (2008), no.~2, 177--222. \MR{2366561}

\bibitem[MFK94]{MFK}
D.~Mumford, J.~Fogarty, and F.~Kirwan, \emph{Geometric invariant theory}, third
  ed., Ergebnisse der Mathematik und ihrer Grenzgebiete (2) [Results in
  Mathematics and Related Areas (2)], vol.~34, Springer-Verlag, Berlin, 1994.
  \MR{1304906}

\bibitem[Nes14]{MR3443257}
A.~Neshitov, \emph{A rigidity theorem for presheaves with
  {$\Omega$}-transfers}, Algebra i Analiz \textbf{26} (2014), no.~6, 78--98.
  \MR{3443257}

\bibitem[Ngu16]{nguyen16}
M.~T. Nguyen, \emph{On the equivariant motivic spectral sequences}, Ph.D.
  thesis, Dissertation, Duisburg, Essen, Universit{\"a}t Duisburg-Essen, 2016.

\bibitem[Pan03]{Pan03}
I.~A. Panin, \emph{The equicharacteristic case of the {G}ersten conjecture},
  Tr. Mat. Inst. Steklova \textbf{241} (2003), no.~Teor. Chisel, Algebra i
  Algebr. Geom., 169--178. \MR{2024050}

\bibitem[PY02]{PY02}
I.~Panin and S.~Yagunov, \emph{Rigidity for orientable functors}, J. Pure Appl.
  Algebra \textbf{172} (2002), no.~1, 49--77. \MR{1904229}

\bibitem[Qui73]{Qui73}
D.~Quillen, \emph{Higher algebraic {$K$}-theory. {I}}, 85--147. Lecture Notes
  in Math., Vol. 341. \MR{0338129}

\bibitem[Ray70]{Ray06}
M.~Raynaud, \emph{Anneaux locaux hens\'eliens}, Lecture Notes in Mathematics,
  Vol. 169, Springer-Verlag, Berlin-New York, 1970. \MR{0277519}

\bibitem[R{\O}06]{RO06}
A.~Rosenschon and P.~A. {\O}stv{\ae}r, \emph{Rigidity for pseudo pretheories},
  Invent. Math. \textbf{166} (2006), no.~1, 95--102. \MR{2242633}

\bibitem[R{\O}08]{MR2399164}
O.~R\"ondigs and P.~A. {\O}stv{\ae}r, \emph{Rigidity in motivic homotopy
  theory}, Math. Ann. \textbf{341} (2008), no.~3, 651--675. \MR{2399164}

\bibitem[Ser68]{Ser68}
J.-P. Serre, \emph{Groupes de {G}rothendieck des sch\'emas en groupes
  r\'eductifs d\'eploy\'es}, Inst. Hautes \'Etudes Sci. Publ. Math. (1968),
  no.~34, 37--52. \MR{0231831}

\bibitem[She78]{Sh78}
C.~C. Sherman, \emph{The {$K$}-theory of an equicharacteristic discrete
  valuation ring injects into the {$K$}-theory of its field of quotients},
  Pacific J. Math. \textbf{74} (1978), no.~2, 497--499. \MR{0480484}

\bibitem[Sus83]{Sus83}
A.~Suslin, \emph{On the {$K$}-theory of algebraically closed fields}, Invent.
  Math. \textbf{73} (1983), no.~2, 241--245. \MR{714090}

\bibitem[Sus03]{Sus03}
\bysame, \emph{On the {G}rayson spectral sequence}, Tr. Mat. Inst. Steklova
  \textbf{241} (2003), no.~Teor. Chisel, Algebra i Algebr. Geom., 218--253.
  \MR{2024054}

\bibitem[SV96]{SV96}
A.~Suslin and V.~Voevodsky, \emph{Singular homology of abstract algebraic
  varieties}, Invent. Math. \textbf{123} (1996), no.~1, 61--94. \MR{1376246}

\bibitem[Tab]{Tab}
G.~Tabuada, \emph{Noncommutative rigidity}, Math. Z. (2017).
  https://doi.org/10.1007/s00209-017-1998-5.

\bibitem[Tho87]{T87}
R.~W. Thomason, \emph{Algebraic {K}-theory of group scheme actions}, Algebraic
  topology and algebraic K-theory (Princeton, NJ, 1983) \textbf{113} (1987),
  539--563.

\bibitem[Voe00a]{MR1764200}
V.~Voevodsky, \emph{Cohomological theory of presheaves with transfers}, Cycles,
  transfers, and motivic homology theories, Ann. of Math. Stud., vol. 143,
  Princeton Univ. Press, Princeton, NJ, 2000, pp.~87--137. \MR{1764200}

\bibitem[Voe00b]{Voev00-1}
\bysame, \emph{Cohomological theory of presheaves with transfers}, Cycles,
  transfers, and motivic homology theories, Ann. of Math. Stud., vol. 143,
  Princeton Univ. Press, Princeton, NJ, 2000, pp.~87--137. \MR{1764200}

\bibitem[Voe00c]{Voev00}
\bysame, \emph{Triangulated categories of motives over a field}, Cycles,
  transfers, and motivic homology theories, Ann. of Math. Stud., vol. 143,
  Princeton Univ. Press, Princeton, NJ, 2000, pp.~188--238. \MR{1764202}

\bibitem[Voe10]{Voev10}
\bysame, \emph{Homotopy theory of simplicial sheaves in completely decomposable
  topologies}, J. Pure Appl. Algebra \textbf{214} (2010), no.~8, 1384--1398.
  \MR{2593670}

\bibitem[Wal85]{W85}
F.~Waldhausen, \emph{Algebraic {K}-theory of spaces}, Algebraic and geometric
  topology, Springer, 1985, pp.~318--419.

\bibitem[Wal96]{Wal96}
M.~E. Walker, \emph{Motivic complexes and the {K}-theory of automorphisms},
  ProQuest LLC, Ann Arbor, MI, 1996, Thesis (Ph.D.)--University of Illinois at
  Urbana-Champaign. \MR{2694778}

\bibitem[Yag11]{MR2854332}
S.~Yagunov, \emph{Remark on rigidity over several fields}, Homology Homotopy
  Appl. \textbf{13} (2011), no.~2, 159--164. \MR{2854332}

\bibitem[Y{\O}09]{YO09}
S.~Yagunov and P.~A. {\O}stv{\ae}r, \emph{Rigidity for equivariant
  {$K$}-theory}, C. R. Math. Acad. Sci. Paris \textbf{347} (2009), no.~23-24,
  1403--1407. \MR{2588790}

\end{thebibliography}
\bibliographystyle{amsalpha}

\end{document}